\documentclass[a4paper,12pt]{amsart}

\usepackage{amsmath}
\usepackage{amssymb}
\usepackage{mathrsfs}
\usepackage{enumerate}
\usepackage{ifthen}
\usepackage{graphicx}
\usepackage[T1]{fontenc} %skandit
\usepackage{color}

\nonstopmode \numberwithin{equation}{section}
\newtheorem{thm}{Theorem}[section]
\newtheorem{cor}{Corollary}[section]
\newtheorem{lem}{Lemma}[section]

\theoremstyle{definition}

\newtheorem{innercustomthm}{Theorem}
\newenvironment{customthm}[1]
  {\renewcommand\theinnercustomthm{#1}\innercustomthm}
  {\endinnercustomthm}

\newcounter{minutes}
\divide\time by 60
\newcounter{hours}
\multiply\time by 60
\addtocounter{minutes}{-\time}
\newcounter {own}
\def\theown {\thesection       .\arabic{own}}

{\qed\bigskip}

\newcounter{alphabet}

\begin{document}

\title{Schwarzian Norm Estimates for Some Classes of Analytic Functions}

\author{Md Firoz Ali}
\address{Md Firoz Ali,
Department of Mathematics, National Institute of Technology Durgapur,
Durgapur- 713209, West Bengal, India.}
\email{ali.firoz89@gmail.com, firoz.ali@maths.nitdgp.ac.in}

\author{Sanjit Pal}
\address{Sanjit Pal,
Department of Mathematics, National Institute of Technology Durgapur,
Durgapur- 713209, West Bengal, India.}
\email{palsanjit6@gmail.com}


\def\thefootnote{}
\footnotetext{ {\tiny File:~\jobname.tex,
printed: \number\year-\number\month-\number\day,
          \thehours.\ifnum\theminutes<10{0}\fi\theminutes }
} \makeatletter\def\thefootnote{\@arabic\c@footnote}\makeatother

\begin{abstract}
Let $\mathcal{A}$ denote the class of analytic functions $f$ in the unit disk $\mathbb{D}=\{z\in\mathbb{C}:|z|<1\}$ normalized by $f(0)=0$, $f'(0)=1$. In the present article, we obtain the sharp estimates of the Schwarzian norm for functions in the classes $\mathcal{G}(\beta)=\{f\in \mathcal{A}:{\rm Re\,}[1+zf''(z)/f'(z)]<1+\beta/2\}$, where $\beta>0$ and $\mathcal{F}(\alpha)=\{f\in \mathcal{A}:{\rm Re\,}[1+zf''(z)/f'(z)]>\alpha\}$, where $-1/2\le \alpha\le 0$. We also establish two-point distortion theorem for functions in the classes $\mathcal{G}(\beta)$ and $\mathcal{F}(\alpha)$.
\end{abstract}

\thanks{}

\maketitle
\pagestyle{myheadings}
\markboth{Md Firoz Ali and Sanjit Pal}{Schwarzian Norm Estimates For Some Classes of Analytic Functions}

\section{Introduction}
Let $\mathcal{H}$ denote the class of analytic functions in the unit disk $\mathbb{D}=\{z\in\mathbb{C}:|z|<1\}$, and  $\mathcal{LU}$ denote the subclass of $\mathcal{H}$ consisting of all locally univalent functions, namely $\mathcal{LU}=\{f\in\mathcal{H}: f'(z)\ne 0, z\in\mathbb{D}\}$. The Schwarzian derivative for a locally univalent function $f\in\mathcal{LU}$ is defined by
$$S_f(z)=\left[\frac{f''(z)}{f'(z)}\right]^{'}-\frac{1}{2}\left[\frac{f''(z)}{f'(z)}\right]^2$$
and the Schwarzian norm (the hyperbolic sup-norm) of $f\in\mathcal{LU}$ is defined by
$$||S_f||=\sup\limits_{z\in\mathbb{D}}(1-|z|^2)^2|S_f(z)|.$$
For a locally univalent function $f\in\mathcal{LU}$, Nehari \cite{Nehari-1949} proved that if $||S_f||\leq 2$ then the function $f$ is univalent in $\mathbb{D}$. On the other hand, for a univalent function $f$, $||S_f||\leq 6$ (see \cite{Nehari-1949}). Both of the constants $2$ and $6$ are best possible.\\

The Schwarzian norm has a significant meaning in the theory of quasiconformal mappings and Teichm\"{u}ller space (see \cite{Lehto-1987}). A mapping $f:\widehat{\mathbb{C}}\rightarrow\widehat{\mathbb{C}}$ of the Riemann sphere $\widehat{\mathbb{C}}=\mathbb{C}\cup\{\infty\}$ is said to be $k$-quasiconformal ($0\le k<1$) mapping if $f$ is a sense preserving homeomorphism of $\widehat{\mathbb{C}}$ and has locally integrable partial derivatives on $\mathbb{C}\setminus\{f^{-1}(\infty)\}$ with $|f_{\bar{z}}|\le k|f_z|$ a.e.. On the other hand, the theory of Teichm\"{u}ler space $\mathcal{T}$ can be identified with the set of Schwarzian derivatives of analytic and univalent functions on $\mathbb{D}$ with quasiconformal extensions to $\widehat{\mathbb{C}}$. It is known that $\mathcal{T}$ is a bounded domain in the Banach space of analytic functions in $\mathbb{D}$ with finite hyperbolic sup-norm (see \cite{Lehto-1987}). Key results connecting the Schwarzian derivative and quasiconformal mappings are given in the following theorem.

\begin{customthm}{A}\label{Thm-q-0001}\cite{Ahlfors-Weill-1962, Kuhnau-1971}
If $f$ extends to a $k$-quasiconformal ($0\le k<1$) mapping of the Riemann sphere $\widehat{\mathbb{C}}$ then $||S_f||\le 6k$. Conversely, if $||S_f||\le 2k$ then $f$ extends to a $k$-quasiconformal mapping of the Riemann sphere $\widehat{\mathbb{C}}$.
\end{customthm}

Let $\mathcal{A}$ denote the class of functions $f$ in $\mathcal{H}$ normalized by $f(0)=0$, $f'(0)=1$. Therefore, every function $f$ in $\mathcal{A}$ has the Taylor series expansion of the form
\begin{equation}\label{q-00001}
f(z)=z+\sum\limits_{n=2}^{\infty}a_nz^n.
\end{equation}
Let $\mathcal{S}$ be the set of all functions $f$ in $\mathcal{A}$ that are univalent in $\mathbb{D}$. A function $f\in\mathcal{A}$ is called starlike (respectively, convex) if the image $f(\mathbb{D})$ is a starlike domain with respect to the origin (respectively, convex). The classes of all starlike and convex functions that are univalent are denoted by $\mathcal{S^*}$ and $\mathcal{C}$, respectively. It is well known that a function $f$ in $\mathcal{A}$ is starlike (respectively, convex) if and only if ${\rm Re\,} [zf'(z)/f(z)] > 0$ (respectively, ${\rm Re\,} [1 + zf''(z)/f'(z)] > 0$) for $z\in\mathbb{D}$.
A function $f$ in $\mathcal{A}$ is said to be starlike of order $\alpha$, $0\le \alpha< 1$ if ${\rm Re\,} [zf'(z)/f(z)] > \alpha$ for $z\in\mathbb{D}$ and is said to be convex of order $\alpha$, $0\le \alpha< 1$ if ${\rm Re\,} [1 + zf''(z)/f'(z)] > \alpha$ for $z\in\mathbb{D}$. The class of all starlike and convex functions of order $\alpha$ is denoted by $\mathcal{S}^*(\alpha)$ and $\mathcal{C}(\alpha)$, respectively. Clearly, a function $f$ in $\mathcal{A}$ belongs to $\mathcal{C}(\alpha)$ if and only if $zf'\in\mathcal{S}^*(\alpha)$. For further information on these classes, we refer to \cite{Duren-1983,Goodman-book-1983}.\\

%Now, we turn to define the classes of functions which is our main concern in this article.
In this article, we are concerned with two different classes of functions $\mathcal{G}(\beta)$, $\beta>0$ and $\mathcal{F}(\alpha)$, $-1/2\le \alpha\le 0$ defined by
$$\mathcal{G}(\beta)=\Bigg\{f\in \mathcal{A}:{\rm Re\,}\left(1+\frac{zf''(z)}{f'(z)}\right)<1+\frac{\beta}{2}\quad\text{for}~z\in\mathbb{D}\Bigg\}$$
and
$$\mathcal{F}(\alpha)=\Bigg\{f\in \mathcal{A}:{\rm Re\,}\left(1+\frac{zf''(z)}{f'(z)}\right)>\alpha\quad\text{for}~z\in\mathbb{D}\Bigg\}.$$
In 1941, Ozaki \cite{Ozaki-1941} introduced the class $\mathcal{G}:=\mathcal{G}(1)$ and proved that functions in $\mathcal{G}$ are univalent in $\mathbb{D}$. Later on, Umezawa \cite{Umezawa-1952} studied the class $\mathcal{G}$ and showed that functions in $\mathcal{G}$ are convex in one direction, i.e. every $f\in\mathcal{G}$ maps $|z|=\rho<r$ for every $\rho$ near $r$ into a contour which may be cut by every straight line parallel to this direction in not more than two points. Moreover, functions in $\mathcal{G}$ are starlike in $\mathbb{D}$ (see \cite{Jovanovic-Obradovic-1995}, \cite{Ponnusamy-Rajasekaran-1995}).
Thus, the class $\mathcal{G}(\beta)$ is included in $\mathcal{S}^*$ whenever $\beta\in (0,1]$. One can easily show that functions in $\mathcal{G}(\beta)$ are not univalent in $\mathbb{D}$ for $\beta>1$. For $0<\beta \le 2/3$, the class $\mathcal{G}(\beta)$ was studied by Uralegaddi {\it et al.} \cite{Uralegaddi-Ganigi-Sarangi-1994}. Later, the full class was studied in \cite{Ali-Vasudevarao-2016, Owa-Nishiwaki-2002, Owa-Srivastava-2002, Ponnusamy-Vasudevarao-2007}. On the other hand, for $\alpha=0$, $\mathcal{F}(0)=:\mathcal{C}$ is the class of convex functions. For $\alpha=-1/2$, functions in the class $\mathcal{F}(-1/2)$ are not necessarily starlike but are convex in some direction and so are close-to-convex. Here, we recall that a function $f\in\mathcal{A}$ is called close-to-convex if $f(\mathbb{D})$ is close-to-convex domain, i.e. the complement of $f(\mathbb{D})$ in $\mathbb{C}$ is the union of closed half lines with pairwise disjoint interiors. Pfaltzgraff et al. \cite{Pfaltzgraff-Reade-Umezawa-1976} proved that $\mathcal{F}(\alpha)$ contains non-starlike functions for all $-1/2 \le \alpha < 0$.\\

Although the fundamental work on Schwarzian  derivative in connection with geometric functions theory have been done in \cite{Ahlfors-Weill-1962, Kuhnau-1971, Nehari-1949}); not much research has been done on Schwarzian  derivative for various subclasses of univalent functions. In connection with Teichm\"{u}ller spaces, it is an interesting problem to estimate the norm of the Schwarzian derivatives for typical subclasses of univalent functions. For the class of convex functions $\mathcal{C}$, the Schwarzian norm satisfies $||S_f||\le 2$ and the estimate is sharp. This result was proved repeatedly by many researchers (see \cite{Lehto-1977, Nehari-1976, Robertson-1969}). In 1996, Suita \cite{Suita-1996} studied the class $\mathcal{C}(\alpha)$, $0\le \alpha\le 1$ and using the integral representation of functions in $\mathcal{C}(\alpha)$ proved that the Schwarzian norm satisfies the following sharp inequality
$$||S_f||\le
\begin{cases}
2 & \text{ if }~ 0\le \alpha\le 1/2,\\
8\alpha(1-\alpha) & \text{ if }~ 1/2< \alpha\le 1.
\end{cases}
$$

A function $f$ in $\mathcal{A}$ is said to be strongly starlike (respectively, strongly convex) of order $\alpha$, $0<\alpha\le 1$ if $|\arg\{zf'(z)/f(z)\}|<\pi\alpha/2$ (respectively, $|\arg\{1+zf''(z)/f'(z)\}|<\pi\alpha/2$) for $z\in\mathbb{D}$. The class of all strongly starlike and strongly convex functions of order $\alpha$ is denoted by $\mathcal{S}^*_{\alpha}$ and $\mathcal{K}_{\alpha}$, respectively. In 1989, Mocanu \cite{Mocanu-1989} proved that $\mathcal{K}_{\gamma(\beta)}\subset \mathcal{S}^*_{\beta}$ for $0<\beta<1$ where
$$
\gamma(\beta):=\frac{2}{\pi} \arctan\left[\tan \frac{\pi\beta}{2} + \frac{\beta}{(1+\beta)^{(1+\beta)/2} (1-\beta)^{(1-\beta)/2}\cos (\pi\beta/2)} \right].
$$
In other words, $\mathcal{K}_{\alpha}\subset \mathcal{S}^*_{\gamma^{-1}(\alpha)}$ for $0<\alpha< 1$, where $\gamma^{-1}$ denotes the
inverse function of $\gamma:[0, 1]\to [0, 1]$. Note that $\gamma(\beta)$ increases from $0$ to $1$ when $\beta$ varies from $0$ to $1$.
For $0<\alpha<1$, Fait {\it et al.} \cite{Fait-Krzyz-Zygmunt-1976} proved that every function $f$ in $\mathcal{S}^*_{\alpha}$ can be extended to a $\sin(\pi\alpha/2)$-quasiconformal mapping of $\widehat{\mathbb{C}}$.
Therefore, by Theorem \ref{Thm-q-0001}, it easily follows that $||S_f||\le 6\sin(\pi\alpha/2)$ which was pointed out by Chiang \cite{Chiang-1991}. By Theorem \ref{Thm-q-0001}, this also implies that a function $f\in\mathcal{K}_{\alpha}$ extends to a $\sin(\pi\gamma^{-1}(\alpha)/2)$-quasiconformal mapping of $\widehat{\mathbb{C}}$ and satisfies $||S_f||\le 6\sin(\pi\gamma^{-1}(\alpha)/2)$. Later, Kanas and Sugawa \cite{Kanas-Sugawa-2011} studied the Schwarzian norm for the class $\mathcal{K}_{\alpha}$ of strongly convex functions by means a different method and proved the sharp inequality $||S_f||\leq 2\alpha$ for $f\in\mathcal{K}_{\alpha}$ which by Theorem \ref{Thm-q-0001}, also shows that every functions in $\mathcal{K}_{\alpha}$ can be extended to an $\alpha$-quasiconformal mapping of $\widehat{\mathbb{C}}$. Here we note that $\sin(\pi\gamma^{-1}(\alpha)/2)<\alpha$ when
$0 <\alpha< 0.3354$ and so Kanas and Sugawa \cite{Kanas-Sugawa-2011} obtained a better bound than in \cite{Fait-Krzyz-Zygmunt-1976} when $\alpha> 0.3355$.\\
%One can easily prove that every functions in $\mathcal{K}_{\alpha}$ extends to an another quasiconformal mapping of $\widehat{\mathbb{C}}$.
%On the other hand, functions in $\mathcal{K}_{\alpha}$ extends to an $\alpha$-quasiconformal mapping of $\widehat{\mathbb{C}}$ and this result gives better bound only when $\alpha>0.3355$ (see \cite{Kanas-Sugawa-2011}). \\

A function $f\in \mathcal{A}$ is said to be uniformly convex function if every circular arc (positively oriented) of the form $\{z\in\mathbb{D}: |z -\xi|=r\}$, $\xi\in\mathbb{D}$, $0<r<|\xi|+1$ is mapped by $f$ univalently onto a convex arc. The class of all uniformly convex functions is denoted by $\mathcal{UCV}$. In particular, $\mathcal{UCV}\subset\mathcal{K}$.
It is well known that (see \cite{Goodman-book-1983}) a function $f\in \mathcal{A}$ is uniformly convex if and only if
$${\rm Re\,}\left(1+\frac{zf''(z)}{f'(z)}\right)> \left|\frac{zf''(z)}{f'(z)}\right|~~\text{for}~ z\in\mathbb{D}.$$
Kanas and Sugawa \cite{Kanas-Sugawa-2011} proved that the Schwarzian norm satisfies $||S_f||\le 8/\pi^2$ for all $f\in\mathcal{UCV}$ and the estimate is sharp. In 2012, Bhowmik and Wirths \cite{Bhowmik-Wirths-2012} studied the class of concave functions $\mathcal{C}o(\alpha)$ for $1\le \alpha\le 2$ and obtained the sharp estimate $||S_f||\le 2(\alpha^2-1)$ for $f\in\mathcal{C}o(\alpha)$. They also proved that $f$ extends to an $(\alpha^2-1)$-quasiconformal mapping of $\widehat{\mathbb{C}}$ when $1\le \alpha<\sqrt{2}$.\\

In this article, our main aim is to find estimates for the modulus of the Schwarzian derivative for functions in $\mathcal{G}(\beta)$, $\beta>0$ and $\mathcal{F}(\alpha)$, $-1/2\le \alpha\le 0$. These results will yield a sharp estimate of the Schwarzian norm for functions in these classes, which will help us to comment on quasiconformal extension of these functions. This also lead us to find a pair of two-point
distortion conditions of such mappings.

\section{Main Results}

%Before going to our main results, first we discuss the region of variability for the Schwarz's functions.
Let $\mathcal{B}$ be the class of analytic functions $\omega:\mathbb{D}\rightarrow\mathbb{D}$ and $\mathcal{B}_0$ be the class of Schwarz functions $\omega\in\mathcal{B}$ with $\omega(0)=0$. The Schwarz's lemma states that a function $\omega\in\mathcal{B}_0$ satisfies $|\omega(z)|\le |z|$ and $|\omega'(0)|\le 1$. The equality occurs in any one of the inequalities if and only if $\omega(z)=e^{i\alpha}z$, $\alpha\in\mathbb{R}$. A natural extension of Schwarz lemma, known as Schwarz-Pick lemma, gives the estimate $|\omega'(z)|\le (1-|\omega(z)|^2)/(1-|z|^2)$, $z\in\mathbb{D}$ when $\omega\in\mathcal{B}$. In 1931, Dieudonn\'{e} \cite{Dieudonne-1931} first obtained the exact region of variability of $\omega'(z_0)$ for a fixed $z_0\in\mathbb{D}$ over the class $\mathcal{B}_0$.

\begin{lem}[Dieudonn\'{e}'s lemma]\cite{Dieudonne-1931, Duren-1983}
Let $\omega\in\mathcal{B}_0$ and $z_0\ne 0$ be a fixed point in $\mathbb{D}$. The region of variability of $\omega'(z_0)$ is given by
\begin{equation}\label{q-00005}
\left|\omega'(z_0)-\frac{\omega(z_0)}{z_0}\right|\le \frac{|z_0|^2-|\omega(z_0)|^2}{|z_0|(1-|z_0|^2)}.
\end{equation}
Moreover, the equality occurs in \eqref{q-00005} if and only if $\omega$ is a Blaschke product of degree $2$ fixing $0$.
\end{lem}

The Dieudonn\'{e}'s lemma is an improvement of the Schwarz's lemma as well as Schwarz-Pick lemma. Here, we remark that a Blaschke product of degree $n\in\mathbb{N}$ is of the form
$$B(z)=e^{i\theta}\prod_{j=1}^{n}\frac{z-z_j}{1-\bar{z_j}z},\quad z,z_j\in\mathbb{D}, ~\theta\in\mathbb{R}.$$
The Dieudonn\'{e}'s lemma will play a crucial role to prove our main results.\\

Before we state our main result, let us recall another important and useful tool known as the differential subordination technique. Many problems in geometric function theory can be solved in a simple and sharp manner with the help of differential subordination.
A function $f\in\mathcal{H}$ is said to be subordinate to another function $g\in\mathcal{H}$ if there exists an analytic function $\omega\in\mathcal{B}_0$ such that $f(z)=g(\omega(z))$ and it is denoted by $f\prec g$. Moreover, when $g$ is univalent, $f\prec g$ if and only if $f(0)=g(0)$ and $f(\mathbb{D})\subset g(\mathbb{D})$. In terms of subordination, the classes $\mathcal{G}(\beta)$ and $\mathcal{F}(\alpha)$ can be defined as:
\begin{equation}\label{q-00010}
f\in\mathcal{G}(\beta)\iff 1+\frac{zf''(z)}{f'(z)}\prec \frac{1-(1+\beta)z}{1-z}
\end{equation}
and
\begin{equation}\label{q-00015}
f\in\mathcal{F}(\alpha)\iff 1+\frac{zf''(z)}{f'(z)}\prec \frac{1+(1-2\alpha)z}{1-z}.
\end{equation}

\begin{thm}\label{Thm-q-0005}
For $\beta>0$, let $f\in\mathcal{G}(\beta)$ be of the form \eqref{q-00001}. Then the Schwarzian derivative and Schwarzian norm satisfy the inequalities
$$|S_f(z)|\le  \frac{\beta(2+\beta)}{2(1-|z|)^2} \quad\text{and}\quad ||S_f||\le 2\beta(2+\beta).$$
Moreover, the equality occurs in both the inequalities for the function $f_0(z)$ defined by
 $$f_0(z)=\frac{1-(1-z)^{1+\beta}}{1+\beta}.$$
\end{thm}

\begin{proof}
For $\beta>0$, let $f\in\mathcal{G}(\beta)$ be of the form \eqref{q-00001}. Then from \eqref{q-00010}, we have
$$1+\frac{zf''(z)}{f'(z)}\prec \frac{1-(1+\beta)z}{1-z}$$
and so there exists an analytic function $\omega:\mathbb{D}\rightarrow\mathbb{D}$ with $\omega(0)=0$ such that
$$1+\frac{zf''(z)}{f'(z)}= \frac{1-(1+\beta)\omega(z)}{1-\omega(z)}.$$
A simple computation gives
$$\frac{f''(z)}{f'(z)}=-\frac{\beta\omega(z)}{z(1-\omega(z))}$$
and therefore,
\begin{align}\label{q-000001}
S_f(z)&=\left[\frac{f''(z)}{f'(z)}\right]^{'}-\frac{1}{2}\left[\frac{f''(z)}{f'(z)}\right]^2\\
&=-\beta\left[\frac{z\omega'(z)-\omega(z)}{z^2(1-\omega(z))^2}+\frac{(2+\beta)\omega^2(z)}{2z^2(1-\omega(z))^2}\right].\nonumber
\end{align}
Let us consider the transformation $\displaystyle\zeta(z)=\omega'(z)-\frac{\omega(z)}{z}$. By \eqref{q-00005}, the function $\zeta$ varies over the closed disk
$$|\zeta|\le \frac{|z|^2-|\omega|^2}{|z|(1-|z|^2)}$$
for fixed $|z|<1$. Using the transformation of $\zeta(z)$ in \eqref{q-000001}, we obtain
$$S_f(z)=-\beta\left[\frac{\zeta(z)}{z(1-\omega(z))^2}+\frac{(2+\beta)\omega^2(z)}{2z^2(1-\omega(z))^2}\right]$$
and hence
\begin{align*}
|S_f(z)|&\le \beta\left[\frac{|\zeta(z)|}{|z||1-\omega(z)|^2}+\frac{(2+\beta)|\omega^2(z)|}{2|z|^2|1-\omega(z)|^2}\right]\\
&\le\beta\left[\frac{|z|^2-|\omega(z)|^2}{|z|^2(1-|z|^2)|1-\omega(z)|^2}+\frac{(2+\beta)|\omega(z)|^2}{2|z|^2|1-\omega(z)|^2}\right].
\end{align*}
For $0<s:=|\omega(z)|\le |z|<1$, we have
\begin{align}\label{q-000005}
|S_f(z)|&\le \beta\left[\frac{|z|^2-s^2}{|z|^2(1-|z|^2)(1-s)^2}+\frac{(2+\beta)s^2}{2|z|^2(1-s)^2}\right]\\
&=\beta~\frac{2|z|^2+s^2(\beta-\beta |z|^2-2|z|^2)}{2|z|^2(1-|z|^2)(1-s)^2}=\beta g(s),\nonumber
\end{align}

where $$g(s)=\frac{2|z|^2+s^2(\beta-\beta |z|^2-2|z|^2)}{2|z|^2(1-|z|^2)(1-s)^2}, \quad 0< s\le |z|.$$
Therefore,
\begin{equation*}
g'(s)=\frac{2|z|^2+s(\beta-\beta |z|^2-2|z|^2)}{|z|^2(1-|z|^2)(1-s)^3}.
\end{equation*}
We claim that $g'(s)>0$ for $0<s\le |z|$. If $\beta-\beta |z|^2-2|z|^2\ge 0$, then clearly $g'(s)>0$. If $\beta-\beta |z|^2-2|z|^2<0$, then using $0< s\le |z|$, we have
\begin{align*}
2|z|^2+s(\beta-\beta |z|^2-2|z|^2)&\ge 2|z|^2+|z|(\beta-\beta |z|^2-2|z|^2)\\
&= |z|(1-|z|)(\beta+2|z|+\beta|z|)>0.
\end{align*}
This lead us to conclude that $g'(s)>0$ for all $0<s\le |z|$. Thus, $g(s)$ attain its maximum at the point $s=|z|$ and so from \eqref{q-000005}, we have
$$|S_f(z)|\le \beta g(|z|)= \frac{\beta(2+\beta)}{2(1-|z|)^2}.$$
Therefore,
\begin{align*}
||S_f||&=\sup\limits_{z\in\mathbb{D}}(1-|z|^2)^2|S_f(z)|\le \frac{\beta(2+\beta)}{2}\sup\limits_{z\in\mathbb{D}}(1+|z|)^2\\
&=2\beta(2+\beta).
\end{align*}
To show that the estimates are sharp, let us consider the function $f_0(z)$ defined by
$$f_0(z)=\frac{1-(1-z)^{1+\beta}}{1+\beta}.$$
A simple computation gives
$$ S_{f_0}(z)=-\frac{\beta (2+\beta)}{2(1-z)^2}$$
and
$$
||S_{f_0}||=\sup\limits_{z\in\mathbb{D}}(1-|z|^2)^2|S_{f_0}(z)|=\beta (2+\beta)\sup\limits_{z\in\mathbb{D}}\frac{(1-|z|^2)^2}{2|1-z|^2}.
$$
On the positive real axis, we note that
$$ \beta (2+\beta)\sup\limits_{0<r<1}\frac{(1-r^2)^2}{2(1-r)^2}=2\beta (2+\beta).$$
Thus, $$||S_{f_0}||=2\beta (2+\beta).$$
\end{proof}

\begin{cor}
If $f\in\mathcal{G}:=\mathcal{G}(1)$ be of the form \eqref{q-00001}, then the Schwarzian norm satisfies the sharp inequality $$||S_f||\le 6.$$
\end{cor}

Theorem \ref{Thm-q-0005} and Theorem \ref{Thm-q-0001} immediately gives the following result for functions in $\mathcal{G}(\beta)$.
\begin{cor}
Let $0<\beta<\sqrt{2}-1$, and $f\in\mathcal{G}(\beta)$ be of the form \eqref{q-00001}. Then $f$ extends to an $\beta(2+\beta)$-quasiconformal mapping.
\end{cor}

\begin{thm}\label{Thm-q-0010}
For $-1/2\le \alpha\le 0$, let $f\in\mathcal{F}(\alpha)$ be of the form \eqref{q-00001}. Then the Schwarzian derivative $S_f(z)$ satisfies the following sharp inequality
$$|S_f(z)|\le \frac{2(1-\alpha)}{(1+\alpha)}\frac{(1+\alpha-\alpha|z|^2)}{(1-|z|^2)^2}\quad\text{for}~z\in\mathbb{D}.$$
Moreover, the equality occurs for some fixed $z_0\in \mathbb{D}$ with $-1<z_0<1$ for the function $f_{z_0}$ defined by
$$1+\frac{zf_{z_0}''(z)}{f_{z_0}'(z)}=\frac{1+(1-2\alpha)\omega_{z_0}(z)}{1-\omega_{z_0}(z)},$$
where
$$\omega_{z_0}(z)=-\frac{z(z-b)}{1-bz}\quad \text{with}\quad b=\frac{z_0(2+\alpha-\alpha z_0^2)}{1+\alpha+z_0^2-\alpha z_0^2}.$$
\end{thm}
\begin{proof}
For $-1/2\le \alpha\le 0$, let $f\in\mathcal{F}(\alpha)$.  Then from \eqref{q-00015}, we have
$$1+\frac{zf''(z)}{f'(z)}\prec \frac{1+(1-2\alpha)z}{1-z}.$$
Thus, there exists an analytic function $\omega:\mathbb{D}\rightarrow\mathbb{D}$ with $\omega(0)=0$ such that
$$1+\frac{zf''(z)}{f'(z)}= \frac{1+(1-2\alpha)\omega(z)}{1-\omega(z)}.$$
A simple computation gives
$$\frac{f''(z)}{f'(z)}=\frac{2(1-\alpha)\omega(z)}{z(1-\omega(z))}$$
and hence,
\begin{align}\label{q-000010}
S_f(z)&=\left[\frac{f''(z)}{f'(z)}\right]^{'}-\frac{1}{2}\left[\frac{f''(z)}{f'(z)}\right]^2\\
&=2(1-\alpha)\frac{\omega'(z)}{z(1-\omega(z))^2}-2(1-\alpha)\frac{\omega(z)-\alpha\omega^2(z)}{z^2(1-\omega(z))^2}.\nonumber
\end{align}
Let us consider the transformation $\displaystyle\zeta(z)=\omega'(z)-\frac{\omega(z)}{z}$. By \eqref{q-00005}, the function $\zeta$ varies over the closed disk
$$|\zeta|\le \frac{|z|^2-|\omega|^2}{|z|(1-|z|^2)}$$
for fixed $|z|<1$. Using the transformation $\zeta(z)$ in  \eqref{q-000010}, we obtain
$$S_f(z)=2(1-\alpha)\frac{\alpha\omega^2(z)}{z^2(1-\omega(z))^2}+2(1-\alpha)\frac{\zeta(z)}{z(1-\omega(z))^2}$$
and consequently,
\begin{align}\label{q-000015}
|S_f(z)|&\le 2(1-\alpha)\frac{-\alpha|\omega(z)|^2}{|z|^2|1-\omega(z)|^2}+2(1-\alpha)\frac{|\zeta(z)|}{|z||1-\omega(z)|^2}\\
&\le 2(1-\alpha)\frac{-\alpha|\omega(z)|^2}{|z|^2|1-\omega(z)|^2}+2(1-\alpha)\frac{|z|^2-|\omega^2(z)|}{|z|^2(1-|z|^2)|1-\omega(z)|^2}.\nonumber
\end{align}
For $0<s:=|\omega(z)|\le |z|<1$, we have
\begin{align}\label{q-000020}
|S_f(z)|&\le 2(1-\alpha)\frac{-\alpha s^2}{|z|^2(1-s)^2}+2(1-\alpha)\frac{|z|^2-s^2}{|z|^2(1-|z|^2)(1-s)^2}\\
&=2(1-\alpha)\frac{|z|^2-s^2(1+\alpha-\alpha |z|^2)}{|z|^2(1-|z|^2)(1-s)^2}=h(s),\nonumber
\end{align}
where $$h(s)=2(1-\alpha)\frac{|z|^2-s^2(1+\alpha-\alpha |z|^2)}{|z|^2(1-|z|^2)(1-s)^2},\quad 0<s\le |z|.$$
Therefore,
\begin{equation*}
h'(s)=\frac{4(1-\alpha)}{|z|^2(1-|z|^2)}\frac{|z|^2-s(1+\alpha-\alpha |z|^2)}{(1-s)^3}.
\end{equation*}
Thus, $h'(s)=0$ implies that $s=|z|^2/(1+\alpha-\alpha |z|^2)$, which lies in $(0,|z|)$ as 
$$\frac{|z|^2}{1+\alpha-\alpha |z|^2}<|z| \iff (1-|z|)(1+\alpha(1+|z|))>0.$$
A simple calculation shows that $h'(0)>0$ and $h'(|z|)<0$. This lead us to conclude that the function $h$ attain its maximum at $|z|^2/(1+\alpha-\alpha |z|^2)$. Consequently, from \eqref{q-000020}, we have
$$|S_f(z)|\le h\left(\frac{|z|^2}{1+\alpha-\alpha |z|^2}\right)= \frac{2(1-\alpha)}{(1+\alpha)}\frac{(1+\alpha-\alpha|z|^2)}{(1-|z|^2)^2}.$$

To show that the estimate \eqref{q-000015} is sharp, let $z_0\in\mathbb{D}$ with $-1<z_0<1$ be fixed and consider the function $f_{z_0}$ defined by
\begin{equation}\label{q-000025}
1+\frac{zf''_{z_0}(z)}{f'_{z_0}(z)}=\frac{1+(1-2\alpha)\omega_{z_0}(z)}{1-\omega_{z_0}(z)},
\end{equation}
where
$$\omega_{z_0}(z)=-\frac{z(z-b)}{1-bz}\quad \text{with}\quad b=\frac{z_0(2+\alpha-\alpha z_0^2)}{1+\alpha+z_0^2-\alpha z_0^2}.$$
Considering $b$ as a function of $z_0$ in $(-1,1)$, we note that for each $-1/2\le \alpha\le 0$,
$$b'(z_0)=\frac{(1-|z_0|^2)\{\alpha^2(1-|z_0|^2)+(2\alpha+1)+(1+\alpha+\alpha|z_0|^2)\}}{(1+\alpha+z_0^2-\alpha z_0^2)^2}>0,~~ z_0\in (-1,1).$$
Also, $b(-1)=-1$ and $b(1)=1$. This lead us to conclude that $|b|<1$ and so $\omega_{z_0}$ is a Blaschke product of degree $2$ fixing $0$. This also shows that $f_{z_0}\in\mathcal{F}(\alpha)$ for every fixed $z_0\in (-1,1)$.\\

For such an $f_{z_0}$, we compute $S_{f_{z_0}}(z_0)$ using  \eqref{q-000010} as follows
\begin{align*}
S_{f_{z_0}}(z_0)
&=\left[\frac{f''_{z_0}(z)}{f'_{z_0}(z)}\right]^{'}-\frac{1}{2}\left[\frac{f''_{z_0}(z)}{f'_{z_0}(z)}\right]^2~\Bigg\vert_{z=z_0}\\
&=2(1-\alpha)\left[\frac{\omega'_{z_0}(z)}{z(1-\omega_{z_0}(z))^2} -\frac{\omega_{z_0}(z) -\alpha\omega_{z_0}^2(z)}{z^2(1-\omega_{z_0}(z))^2}\right] ~\Bigg\vert_{z=z_0}\\
&=-\frac{2(1-\alpha)}{(1+\alpha)}\frac{(1+\alpha-\alpha z_0^2)}{(1-z_0^2)^2}.
\end{align*}
Therefore,
\begin{equation}\label{q-000030}
|S_{f_{z_0}}(z_0)|=\frac{2(1-\alpha)}{(1+\alpha)}\frac{(1+\alpha-\alpha z_0^2)}{(1-z_0^2)^2}.
\end{equation}
This completes the proof.
\end{proof}

\begin{thm}\label{Thm-q-0015}
For $-1/2\le \alpha\le 0$, let $f\in\mathcal{F}(\alpha)$ be of the form \eqref{q-00001}. Then the Schwarzian norm $||S_f(z)||$ satisfies
$$||S_f||\le \frac{2(1-\alpha)}{1+\alpha}$$
and the estimate is best possible.
\end{thm}
\begin{proof}
For $f\in\mathcal{F}(\alpha)$, from Theorem \ref{Thm-q-0010}, we have
$$|S_f(z)|\le \frac{2(1-\alpha)}{(1+\alpha)}\frac{(1+\alpha-\alpha|z|^2)}{(1-|z|^2)^2}.$$
Therefore,
\begin{align*}
||S_f||&=\sup\limits_{z\in\mathbb{D}}(1-|z|^2)^2|S_f(z)|\\
&\le\frac{2(1-\alpha)}{(1+\alpha)} \sup\limits_{z\in\mathbb{D}}(1+\alpha-\alpha|z|^2)=\frac{2(1-\alpha)}{1+\alpha}.
\end{align*}
To show that the estimate is best possible, we consider the function $f_{z_0}(z)$, $-1<z_0<1$ defined by \eqref{q-000025}. Then from \eqref{q-000030}, we have
$$(1-|z_0|^2)^2|S_{f_{z_0}}(z_0)|=\frac{2(1-\alpha)(1+\alpha-\alpha z_0^2)}{(1+\alpha)}\rightarrow \frac{2(1-\alpha)}{1+\alpha}\quad\text{as}~z_0\rightarrow 1^-.$$
This shows that the estimate is best possible.\\[2mm]
\end{proof}

If we choose $\alpha=0$ in Theorem \ref{Thm-q-0015}, we get the Schwarzian norm  estimate for the class of convex functions, which was first proved by Nehari \cite{Nehari-1976}.
\begin{cor}
If $f\in\mathcal{F}(0)=:\mathcal{C}$ be of the form \eqref{q-00001}, then the Schwarzian norm satisfies the sharp inequality $$||S_f||\le 2.$$
\end{cor}

For $z_1,z_2\in\mathbb{D}$, let the hyperbolic metric $\lambda(z_1,z_2)$ be defined by
$$\lambda(z_1,z_2)=\frac{1}{2}\log\frac{1+\rho(z_1,z_2)}{1-\rho(z_1,z_2)},\quad\text{where}~ \rho(z_1,z_2)=\left|\frac{z_1-z_2}{1-\bar{z_1}z_2}\right|.$$
We also define the following quantity for an analytic and locally univalent function $f$ in $\mathbb{D}$:
$$\Delta_f(z_1,z_2):=\frac{|f(z_1)-f(z_2)|}{(1-|z_1|^2)(1-|z_2|^2)\sqrt{|f'(z_1)||f'(z_2)|}},\quad z_1,z_2\in \mathbb{D}.$$

In order to prove the two point distortion theorem for the classes $\mathcal{G}(\beta)$ and $\mathcal{F}(\alpha)$, we need to state the most classical result about two point distortion, which was proved by Chuaqui {\it et al.} \cite{Chuaqui-Duren-Ma-Mejia-Minda-Osgood-2011}.

\begin{customthm}{B}\label{Thm-q-0020}\cite[Theorem 1.]{Chuaqui-Duren-Ma-Mejia-Minda-Osgood-2011}
Let $f$ be analytic and locally univalent in $\mathbb{D}$ and suppose that the bound $||S_f||\le 2(1+\delta^2)$ holds for some $\delta>0$. Then
\begin{equation}\label{q-000035}
\Delta_f(z_1,z_2)\ge \frac{1}{\delta}\sin\left(\delta\lambda(z_1,z_2)\right)
\end{equation}
for all $z_1,z_2\in\mathbb{D}$ with $\lambda(z_1,z_2)\le \pi/\delta$, and
\begin{equation}\label{q-000040}
\Delta_f(z_1,z_2)\le \frac{1}{\sqrt{2+\delta^2}}\sinh\left(\sqrt{2+\delta^2}~\lambda(z_1,z_2)\right)
\end{equation}
for all $z_1,z_2\in\mathbb{D}$. Each of the inequalities \eqref{q-000035} and \eqref{q-000040} is sharp; for each pair of points $z_1$ and $z_2$ in the specified range, equality occurs for some function $f$ with $||S_f||\le 2(1+\delta^2)$. Equality holds in \eqref{q-000035} precisely for $f=T\circ F\circ \sigma$ and in \eqref{q-000040} for $f=T\circ G\circ \sigma$, where $F$ and $G$ are defined by
\begin{equation}\label{q-000045}
F(z)=\left(\frac{1+z}{1-z}\right)^{i\delta}\quad\text{and}\quad G(z)=\left(\frac{1+z}{1-z}\right)^{\sqrt{2+\delta^2}},
\end{equation}
$\sigma$ is the M\"{o}bius automorphism of $\mathbb{D}$ with $\sigma(z_1)=0$ and $\sigma(z_2)>0$, and $T$ is an arbitrary M\"{o}bius transformation. For each such function $f$, equality holds along the entire (admissible portion of the) hyperbolic geodesic through $z_1$ and $z_2$. Conversely, if either inequality holds for all points $z_1$ and $z_2$ in the specified range, then $||S_f||\le 2(1+\delta^2)$.
\end{customthm}

Now, in view of Theorem \ref{Thm-q-0005}, Theorem \ref{Thm-q-0015} and Theorem \ref{Thm-q-0020}, we obtain a pair of two-point distortion theorems for functions in $\mathcal{G}(\beta)$ and $\mathcal{F}(\alpha)$ for certain range of $\beta$ and $\alpha$:
\begin{cor}
Let $\beta>\sqrt{2}-1$ and $f\in\mathcal{G}(\beta)$ be of the form \eqref{q-00001}. Then
$$\Delta_f(z_1,z_2)\ge \frac{1}{\sqrt{\beta^2+2\beta-1}}\sin\left(\sqrt{\beta^2+2\beta-1}~\lambda(z_1,z_2)\right)$$
for all $z_1,z_2\in\mathbb{D}$ with $\lambda(z_1,z_2)\le \pi/\sqrt{\beta^2+2\beta-1}$ and
$$\Delta_f(z_1,z_2)\le \frac{1}{1+\beta}\sinh\left((1+\beta)\lambda(z_1,z_2)\right)$$
for all $z_1,z_2\in\mathbb{D}$. Both inequalities are sharp.
\end{cor}
\begin{cor}
Let $-1/2\le \alpha<0$ and $f\in\mathcal{F}(\alpha)$ be of the form \eqref{q-00001}. Then
$$\Delta_f(z_1,z_2)\ge \sqrt{\frac{1+\alpha}{-2\alpha}}\sin\left(\sqrt{\frac{-2\alpha}{1+\alpha}}~\lambda(z_1,z_2)\right)$$
for all $z_1,z_2\in\mathbb{D}$ with $\lambda(z_1,z_2)\le \pi\sqrt{(1+\alpha)/(-2\alpha)}$ and
$$\Delta_f(z_1,z_2)\le \sqrt{\frac{1+\alpha}{2}}\sinh\left(\sqrt{\frac{2}{1+\alpha}}~\lambda(z_1,z_2)\right)$$
for all $z_1,z_2\in\mathbb{D}$. Both inequalities are sharp.
\end{cor}

%\begin{cor}
%Let $-1/2\le \alpha\le 0$ and $f\in\mathcal{F}(\alpha)$. Then $f$ extends to an $(1-\alpha)/(1+\alpha)$-quasiconformal mapping.
%\end{cor}

\vspace{4mm}
\textbf{Data availability:} Data sharing not applicable to this article as no data sets were generated or analyzed during the current study.\\

\textbf{Authors contributions:}  All authors contributed equally to the investigation of the problem and the order of the authors is given alphabetically according to the surname. All authors read and approved the final manuscript.\\

\textbf{Acknowledgement:} The second named author thanks the University Grants Commission for the financial support through UGC Fellowship (Grant No. MAY2018-429303).

\end{document}